\documentclass[final,onefignum,onetabnum,onealgnum]{siamart251216}

\usepackage{amsmath,amsfonts,amsopn,amssymb}
\usepackage{graphicx}
\usepackage[caption=false]{subfig}
\usepackage{multirow,relsize,makecell}
\usepackage{url}
\ifpdf
\DeclareGraphicsExtensions{.eps,.pdf,.png,.jpg}
\else
\DeclareGraphicsExtensions{.eps}
\fi
\usepackage{algorithm,algorithmic}
\usepackage[shortlabels]{enumitem}
\usepackage{accents}

\numberwithin{theorem}{section}
\newsiamremark{assumption}{Assumption}
\newsiamremark{remark}{Remark}
\newsiamremark{example}{Example}
\crefname{assumption}{Assumption}{Assumptions}
\crefname{remark}{Remark}{Remarks}
\crefname{example}{Example}{Examples}
\AddToHook{env/assumption/begin}{\crefalias{theorem}{assumption}}
\AddToHook{env/remark/begin}{\crefalias{theorem}{remark}}

\headers{Two-level preconditioning for reflections}{Jongho Park}
\title{Symmetric two-level preconditioning\\
for the method of reflections\\
for Laplace and Stokes problems\thanks{Submitted to arXiv.
\funding{This work was supported by the University Development Fund from The Chinese University of Hong Kong,
Shenzhen (UDF01004630).}
}}
\author{Jongho Park\thanks{School of Science and Engineering, The Chinese University of Hong Kong, Shenzhen, Guangdong 518172, China (\email{jonghopark@cuhk.edu.cn}).}}
\ifpdf
\hypersetup{pdftitle={Symmetric two-level preconditioning for the method of reflections for Laplace and Stokes problems},pdfauthor={Jongho Park}}
\fi
\graphicspath{{figures/}}

\begin{document}
\maketitle

\begin{abstract}
We develop and analyze symmetric two-level preconditioners for the method of reflections for Laplace and Stokes problems in three-dimensional perforated domains. Both models admit quadratic constrained minimization formulations, with the Stokes velocity restricted to divergence-free functions. We combine parallel reflection corrections with a coarse correction based on constant boundary modes to obtain symmetric positive definite preconditioners for the conjugate gradient method. Under a geometric separation condition, we prove condition number bounds independent of the number and sizes of the holes and the diameter of the ambient domain, with explicit dependence on the separation parameter. The coarse problem has one unknown per hole for Laplace problems and three for Stokes problems. We also derive convergence estimates for the one-level methods.
\end{abstract}


\section{Introduction}
\label{Sec:Introduction}
We construct and analyze a symmetric two-level preconditioner for the method of reflections for Laplace and Stokes problems on perforated domains in $\mathbb R^3$. The method of reflections is widely used in potential theory, fluid mechanics, and multiple scattering problems, and its convergence and acceleration have received increasing attention in recent years~\cite{CDGH:2026a,CDGH:2026b,CGHS:2019,HV:2018,LLS:2021}.

The method of reflections solves boundary value problems by combining auxiliary solves associated with individual boundary components. Each correction enforces the prescribed condition on one component, using the current iterate to account for the interactions with the others. The parallel version was described in~\cite{Golusin:1934}. As the number of holes increases, the collective effect of the interactions can slow convergence, motivating the use of a coarse correction~\cite{CGHS:2019}.

The projection formulation in~\cite{LLS:2021} establishes convergence of alternating reflections and an averaged parallel method in the orthogonal setting. Coarse corrections and two-level methods of reflections were introduced in~\cite{CGHS:2019}, using low-order boundary modes in a formulation that identifies reflections with block Gauss--Seidel and block Jacobi iterations. Their scalability was investigated numerically.
More recently, boundary integral formulations have been used to interpret the methods of reflections as algebraic Schwarz methods for Helmholtz multiple scattering problems~\cite{CDGH:2026a}. This interpretation leads to accelerated variants incorporating overlap and coarse corrections, which can also be used as preconditioners for GMRES. Related constructions for time-harmonic elasticity are presented in~\cite{CDGH:2026b}. 

Building on these developments, we construct symmetric positive definite two-level preconditioners for Laplace and Stokes problems using constant boundary modes and prove a condition number bound independent of the number and sizes of the holes and the domain diameter under an explicit geometric separation condition.

The common formulation is obtained by restricting the Stokes velocity to functions satisfying $\operatorname{div}u=0$. The velocity is then the minimizer of a quadratic energy subject to the prescribed traces, just as in the Laplace problem. This variational treatment of Stokes problems is classical~\cite{Galdi:2011}, and the use of divergence-free spaces in reflection analysis is developed in~\cite{Hofer:2021,HV:2018}. We use this viewpoint to obtain a common symmetric coercive dual formulation and a common two-level construction. The dualization framework in~\cite{JPX:2025} identifies alternating and parallel reflections with the block Gauss--Seidel and block Jacobi methods for this formulation. The corresponding primal and dual errors have equal energy norms, allowing convergence estimates for the block methods to be transferred to the reflection algorithms. We first derive one-level estimates that display the dependence on the number of holes. For fixed relative separation, the condition number of the block Jacobi method is bounded by $CN^{\frac{2}{3}}$. The analysis uses spectral estimates for the block Jacobi method and the Xu--Zikatanov identity~\cite{XZ:2002} for the block Gauss--Seidel method. Numerical comparisons illustrate the influence of block ordering on successive corrections.

For the Laplace problem, our coarse space agrees with the constant-mode space in~\cite{CGHS:2019} after the block diagonal change of variables. For the Stokes problem, we use three constant vector traces on each boundary component. We define the coarse correction through the dual energy, so that it gives an orthogonal projection, and apply it before and after the parallel local corrections to preserve symmetry. The key estimates control dual variables that annihilate constant traces on each boundary component using disjoint enclosing neighborhoods. Together with subspace correction arguments~\cite{PX:2025,TW:2005}, these estimates give a condition number bound independent of the number of holes. We also track the separation parameter $0<\beta\le1$ explicitly in the hybrid condition number bounds. We describe the corresponding two-level correction directly in the primal variable. Spherical-harmonic Galerkin experiments for both models support the analysis. 

The remainder of the paper is organized as follows. We introduce the two model problems and their reflection algorithms in \cref{Sec:Reflections}, and derive their common dual formulation in \cref{Sec:Duality}. Geometric assumptions and one-level convergence estimates are presented in \cref{Sec:OneLevel}. The two-level method and its analysis are developed in \cref{Sec:TwoLevel}. Galerkin discretization and numerical results for both models are discussed in \cref{Sec:Computational}. We conclude in \cref{Sec:Conclusion}.

\subsection{Notation}
\label{Subsec:notation}
For real Hilbert spaces $V$ and $W$, the topological dual of $V$ is denoted by $V^*$, and the duality pairing is denoted by $\langle\cdot,\cdot\rangle\colon V^*\times V\to\mathbb R$.
Hilbert spaces are identified with their double duals. The dual adjoint of a continuous linear operator $B\colon V\to W$ is denoted by $B^*\colon W^*\to V^*$.
For matrices, the superscript $t$ denotes the transpose, and the duality pairing reduces to the Euclidean inner product.

A continuous linear operator $C\colon V\to V^*$ is called symmetric positive definite (SPD) if the bilinear form $\langle Cv,w\rangle$ is symmetric and coercive. The associated inner product, norm, and induced operator norm are defined by
\[
(v,w)_C=\langle Cv,w\rangle,
\quad \|v\|_C^2=\langle Cv,v\rangle,
\quad \|T\|_C=\sup_{0\ne v\in V}\frac{\|Tv\|_C}{\|v\|_C},
\]
where $T\colon V\to V$ is continuous and linear.
The spectrum of $T$ is denoted by $\sigma(T)$. If $T$ is self-adjoint and coercive, its spectral condition number is defined by
\begin{equation*}
\kappa(T)=\frac{\sup\sigma(T)}{\inf\sigma(T)}.
\end{equation*}

For nonnegative quantities $x$ and $y$, we write $x\lesssim y$ if $x\le Cy$ for a constant independent of the parameters of interest.
We use $x\gtrsim y$ for the reverse inequality and $x\eqsim y$ when both hold.

\section{The method of reflections}
\label{Sec:Reflections}
In this section, we introduce Laplace and Stokes problems in perforated domains and review their reflection algorithms. The Laplace formulation follows~\cite{CGHS:2019}, while reflection methods for Stokes problems are discussed in~\cite{Hofer:2021,HV:2018}. We use a common variational formulation to describe the local corrections and their projection interpretation.

\subsection{Model problems}
Let $\Omega\subset\mathbb R^3$ be a bounded connected Lipschitz domain containing $N$ connected Lipschitz subdomains $O_i$ with pairwise disjoint closures. Each $\Gamma_i=\partial O_i$ has positive distance from the other holes and from $\partial\Omega$. We assume that the perforated domain
\[
\Omega_{\mathrm p}=\Omega\setminus\bigcup_{i=1}^N\overline O_i
\]
is connected. The subscript $\mathrm p$ stands for \emph{perforated}. We consider Laplace and Stokes problems on this domain.

The Laplace boundary value problem is given by
\begin{equation}
\label{model}
\begin{aligned}
-\Delta u&=0&&\text{in }\Omega_{\mathrm p},\\
u&=0&&\text{on }\partial\Omega,\\
u&=g_i&&\text{on }\Gamma_i,\quad 1\le i\le N,
\end{aligned}
\end{equation}
where $g_i\in H^{\frac{1}{2}}(\Gamma_i)$. By extending the solution harmonically into each $O_i$ with boundary value $g_i$, we obtain a function in $H_0^1(\Omega)$. The matching interior and exterior traces give a globally $H^1$ function; see~\cite[Theorem~18.8]{EG:2021}. We use the same symbol $u$ for this extension. For this model, we set
\begin{equation}
\label{laplace_spaces}
V=H_0^1(\Omega),
\quad W_i=H^{\frac{1}{2}}(\Gamma_i),
\quad a(v,w)=\int_\Omega\nabla v\cdot\nabla w\,dx.
\end{equation}

The Stokes boundary value problem is given by
\begin{equation}
\label{model_stokes}
\begin{aligned}
-\Delta u+\nabla p&=0&&\text{in }\Omega_{\mathrm p},\\
\operatorname{div}u&=0&&\text{in }\Omega_{\mathrm p},\\
u&=0&&\text{on }\partial\Omega,\\
u&=g_i&&\text{on }\Gamma_i,\quad 1\le i\le N,
\end{aligned}
\end{equation}
where $u$ is vector-valued and $p$ is defined up to a constant. Writing $n_i$ for the outward unit normal of $O_i$, we assume the compatibility conditions
\begin{equation}
\label{stokes_compatibility}
g_i\in H^{\frac{1}{2}}(\Gamma_i)^3,
\quad \int_{\Gamma_i}g_i\cdot n_i\,ds=0,
\quad 1\le i\le N.
\end{equation}
These conditions permit a divergence-free extension into each $O_i$. They include constant vector data and traces of rigid motions. Although solvability on $\Omega_{\mathrm p}$ alone only requires a condition on the sum of the boundary integrals, we impose~\eqref{stokes_compatibility} to obtain the extension on the whole domain used by the reflection formulation.

By solving the homogeneous Stokes problem in each $O_i$ with trace $g_i$, we extend $u$ to the real Hilbert space
\begin{subequations}
\label{stokes_spaces}
\begin{equation}
V=\{v\in H_0^1(\Omega)^3:\operatorname{div}v=0\},
\quad
W_i=\left\{w_i\in H^{\frac{1}{2}}(\Gamma_i)^3:
\int_{\Gamma_i}w_i\cdot n_i\,ds=0\right\}.
\end{equation}
For this model, the bilinear form is defined by
\begin{equation}
a(v,w)=\int_\Omega\nabla v:\nabla w\,dx,
\quad v,w\in V.
\end{equation}
\end{subequations}
The existence of the divergence-free extensions and the variational characterization of~\eqref{model_stokes} follow from~\cite{Galdi:2011}.

In the following, $V$, $W_i$, and $a(\cdot,\cdot)$ refer to the spaces and bilinear form of the selected model; see~\eqref{laplace_spaces} and~\eqref{stokes_spaces}. In both cases, $a(\cdot,\cdot)$ is an inner product on $V$, and the continuous trace operator is denoted by
\[
\gamma_i\colon V\to W_i,
\quad \gamma_i v=v|_{\Gamma_i}.
\]
For the Laplace problem, we have $W_i^*=H^{-\frac{1}{2}}(\Gamma_i)$. For the Stokes problem, the dual space is identified with the quotient
\begin{equation}
\label{stokes_quotient}
W_i^*=H^{-\frac{1}{2}}(\Gamma_i)^3/\operatorname{span}\{n_i\}.
\end{equation}
Indeed, adding a multiple of $n_i$ does not change the action on $W_i$~\cite{SS:2014}. Pairings with constant vectors are well defined on this quotient because $\int_{\Gamma_i}n_i\,ds=0$.

\Cref{Table:model_spaces} summarizes the model-dependent meanings of the notation used in the remainder of the paper. The operators and algorithms in the following sections have the same definitions for both choices.

\begin{table}
\centering
\begin{tabular}{c|c|c}
 & Laplace & Stokes \\ \hline
$V$ & $H_0^1(\Omega)$ & $\{v\in H_0^1(\Omega)^3:\operatorname{div}v=0\}$ \\
$W_i$ & $H^{\frac{1}{2}}(\Gamma_i)$ & $\{w\in H^{\frac{1}{2}}(\Gamma_i)^3:\int_{\Gamma_i}w\cdot n_i\,ds=0\}$ \\
$W_i^*$ & $H^{-\frac{1}{2}}(\Gamma_i)$ & $H^{-\frac{1}{2}}(\Gamma_i)^3/\operatorname{span}\{n_i\}$ \\
$a(v,w)$ & $\displaystyle\int_\Omega\nabla v\cdot\nabla w\,dx$ & $\displaystyle\int_\Omega\nabla v:\nabla w\,dx$
\end{tabular}
\caption{Spaces and bilinear forms for the two model problems. In both cases, $W=\prod_{i=1}^NW_i$ and $\gamma_i$ denotes the trace on $\Gamma_i$.}
\label{Table:model_spaces}
\end{table}

\subsection{Local extensions}
We use the product trace space $W=\prod_{i=1}^NW_i$ and the combined trace operator $\gamma v=(\gamma_1v,\ldots,\gamma_Nv)$. Define the spaces of homogeneous extensions by
\begin{equation*}
H=(\ker\gamma)^{\perp_a},
\quad H_i=(\ker\gamma_i)^{\perp_a}.
\end{equation*}
For the Laplace problem, $H$ consists of functions harmonic on $\Omega_{\mathrm p}$ and each $O_i$, whereas $H_i$ consists of functions harmonic on $O_i$ and $\Omega\setminus\overline O_i$. For the Stokes problem, these statements hold with the homogeneous Stokes equations in place of the Laplace equation. All equations are understood weakly, and matching traces follow from membership in $V$.

For $w_i\in W_i$, let $\mathcal H_iw_i$ be the unique element of $H_i$ with trace $w_i$ on $\Gamma_i$. The local extension operator satisfies
\begin{equation}
\label{harmonic_extension}
\mathcal H_i\colon W_i\to H_i,
\quad \gamma_i\mathcal H_iw_i=w_i.
\end{equation}
It is computed by solving the corresponding Dirichlet problems in $O_i$ and $\Omega\setminus\overline O_i$. In the Stokes case, both solves include the divergence-free constraint.

\subsection{Alternating and parallel reflections}
The method of reflections corrects boundary discrepancies by the local extensions in~\eqref{harmonic_extension}. The corrections are applied successively or in parallel, as described in \cref{Alg:alternating_reflections,Alg:parallel_reflections}. In the parallel method, the relaxation parameter $\omega$ is chosen to ensure convergence; for example, $\omega=\frac{1}{N}$ gives the averaged method analyzed in~\cite[Theorem~2]{LLS:2021}.

\begin{algorithm}
\caption{Alternating method of reflections}
\label{Alg:alternating_reflections}
\begin{algorithmic}
\STATE Choose $u^{(0)}\in H$.
\FOR{$n=0,1,2,\ldots$}
  \FOR{$i=1,2,\ldots,N$}
    \STATE Compute the correction using $\mathcal H_i$ (see~\eqref{harmonic_extension}) and update the iterate:
\begin{align*}
d_i^{(n)}&=\mathcal{H}_i\bigl(g_i-\gamma_i u^{(n+\frac{i-1}{N})}\bigr),\\
u^{(n+\frac{i}{N})}&=u^{(n+\frac{i-1}{N})}+d_i^{(n)}.
\end{align*}
  \ENDFOR
\ENDFOR
\end{algorithmic}
\end{algorithm}

\begin{algorithm}
\caption{Parallel method of reflections}
\label{Alg:parallel_reflections}
\begin{algorithmic}
\STATE Choose $u^{(0)}\in H$ and $\omega>0$.
\FOR{$n=0,1,2,\ldots$}
  \FOR{$i=1,2,\ldots,N$ \textbf{in parallel}}
    \STATE Compute the correction using $\mathcal H_i$ (see~\eqref{harmonic_extension}):
\begin{equation*}
d_i^{(n)}=\mathcal{H}_i(g_i-\gamma_i u^{(n)}).
\end{equation*}
  \ENDFOR
  \STATE Update the iterate:
\begin{equation*}
u^{(n+1)}=u^{(n)}+\omega\sum_{i=1}^N d_i^{(n)}.
\end{equation*}
\ENDFOR
\end{algorithmic}
\end{algorithm}
The projection interpretation in~\cite{LLS:2021} identifies $\mathcal{H}_i\gamma_i$ as the $a(\cdot,\cdot)$-orthogonal projection onto $H_i$. Indeed, $v-\mathcal{H}_i\gamma_i v$ has zero trace on $\Gamma_i$, and the definition of $H_i$ gives
\[
a(v-\mathcal{H}_i\gamma_i v,w_i)=0,\quad w_i\in H_i.
\]
Thus, on the error, alternating reflections are a product of complementary orthogonal projections~\cite{XZ:2002}.

\subsection{Relevant existing results}
We summarize results from recent works on the projection interpretation, block iterative formulations, and coarse corrections that are directly relevant to the present work.

The Hilbert space analysis in~\cite[Theorems~1 and~2]{LLS:2021} uses orthogonal projections to establish convergence of alternating reflections and of the averaged parallel method. For alternating reflections, the same work discusses subspace-angle estimates and the Xu--Zikatanov identity~\cite{XZ:2002} for the associated product of projections.

For the Laplace problem, the boundary formulation in~\cite{CGHS:2019} uses the traces on $\Gamma_i$ of components in $H_i$. Alternating and parallel reflections are identified with the block Gauss--Seidel and block Jacobi methods for this boundary system, respectively. Based on the same boundary system, coarse corrections are constructed from low-order boundary modes and combined with either block iteration. The scalability of the resulting two-level methods is investigated numerically.

\section{Convex duality and the symmetric dual formulation}
\label{Sec:Duality}
In this section, we derive a common symmetric dual formulation of~\eqref{model} and~\eqref{model_stokes} and identify the methods of reflections as dualizations~\cite{JPX:2025} of block iterative methods.

\subsection{Convex duality}
Fenchel--Rockafellar duality relates a convex minimization problem to a second problem on a dual space~\cite{ET:1999,Rockafellar:1967}. Let $V$ and $W$ be real Hilbert spaces, let $F\colon V\to\overline{\mathbb{R}}$ and $G\colon W\to\overline{\mathbb{R}}$ be proper, lower semicontinuous convex functions, and let $B\colon V\to W$ be a continuous linear operator. The Legendre--Fenchel conjugate $F^*\colon V^*\to\overline{\mathbb{R}}$ is defined by
\[
F^*(\ell)=\sup_{v\in V}\{\langle\ell,v\rangle-F(v)\},
\quad \ell\in V^*.
\]
The conjugate $G^*$ is defined in the same way. Fenchel--Rockafellar duality associates the primal problem
\begin{equation}
\label{general_primal}
\min_{v\in V}\{F(v)+G(Bv)\}
\end{equation}
with the dual problem
\begin{equation}
\label{general_dual}
\min_{\mu\in W^*}\{F^*(-B^*\mu)+G^*(\mu)\}.
\end{equation}
We assume that both problems admit solutions and that strong duality holds; sufficient conditions are given in~\cite{ET:1999,Rockafellar:1967}. Then $u\in V$ and $\lambda\in W^*$ solve~\eqref{general_primal} and~\eqref{general_dual}, respectively, if and only if they satisfy the primal--dual relations
\begin{equation}
\label{primal_dual}
-B^*\lambda\in\partial F(u),
\quad Bu\in\partial G^*(\lambda),
\end{equation}
where $\partial$ denotes the convex subdifferential.

An iterative method $\{u^{(n)}\}$ for the primal problem~\eqref{general_primal} is called a \emph{dualization}~\cite{JPX:2025} of an iterative method $\{\lambda^{(n)}\}$ for the dual problem~\eqref{general_dual} if, for any initial configuration of the dual method, the primal method can be initialized so that their iterates satisfy
\begin{equation*}
-B^*\lambda^{(n)}\in\partial F(u^{(n)}),
\quad n\ge1.
\end{equation*}
When $F^*$ is differentiable at $-B^*\lambda^{(n)}$, this relation determines the primal iterate by $u^{(n)}=(F^*)'(-B^*\lambda^{(n)})$. In the quadratic case below, this recovery is linear.

\subsection{A constrained energy and its dual}
We now express both model problems as constrained minimization problems on their respective spaces $V$. With $W=\prod_{i=1}^NW_i$ and $g=(g_1,\ldots,g_N)\in W$, we define the energy and its associated operator by
\[
F(v)=\frac12a(v,v),
\quad A\colon V\to V^*,
\quad \langle Av,w\rangle=a(v,w),\quad v,w\in V.
\]
The operator $A$ is SPD, and $\|v\|_A=|v|_{H^1(\Omega)}$, with the componentwise seminorm in the Stokes case.

The extended solution of either model problem is characterized by the following minimization problem:
\begin{equation}
\label{primal_energy}
\min_{v\in V}\left\{\frac12a(v,v)+\chi_{\{g\}}(\gamma v)\right\},
\end{equation}
where $\chi_{\{g\}}$ is the indicator function, equal to zero at $g$ and $+\infty$ elsewhere. To write~\eqref{primal_energy} in the form~\eqref{general_primal}, take $B=-\gamma$ and $G=\chi_{\{-g\}}$. The conjugate formulas
\[
G^*(\mu)=-\langle \mu,g\rangle,
\quad F^*(\ell)=\frac12\langle\ell,A^{-1}\ell\rangle
\]
give the dual problem
\begin{equation}
\label{dual_energy}
\min_{\mu\in W^*}\left\{\frac12\langle \mu,S\mu\rangle-\langle \mu,g\rangle\right\},
\end{equation}
where $S=\gamma A^{-1}\gamma^*\colon W^*\to W$. Its optimality condition is the linear equation
\begin{equation}
\label{dual_equation}
S\lambda=g.
\end{equation}
Since $W=\prod_{i=1}^NW_i$ has a product structure, $S$ has blocks
\begin{subequations}
\label{S_blocks}
\begin{equation}
S_{ij}=\gamma_iA^{-1}\gamma_j^*\colon W_j^*\to W_i,
\quad 1\le i,j\le N.
\end{equation}
Namely, we have
\begin{equation}
S=\begin{bmatrix}
S_{11}&\cdots&S_{1N}\\
\vdots&\ddots&\vdots\\
S_{N1}&\cdots&S_{NN}
\end{bmatrix}.
\end{equation}
\end{subequations}
For the Laplace problem, let $G_\Omega$ denote the Dirichlet Green function for $-\Delta$ on $\Omega$. For $\mu_i\in L^2(\Gamma_i)$ and $r_j\in L^2(\Gamma_j)$, the Green representation~\cite{McLean:2000} gives
\begin{equation*}
\langle \mu_i,S_{ij}r_j\rangle
=\int_{\Gamma_i}\int_{\Gamma_j}G_\Omega(x,y)\mu_i(x)r_j(y)\,ds_y\,ds_x.
\end{equation*}
By the first primal--dual relation in~\eqref{primal_dual}, the primal solution is $u=u(\lambda)$, where the primal variable associated with a dual variable $\mu\in W^*$ is given by
\begin{equation}
\label{linear_recovery}
u(\mu)=A^{-1}\gamma^*\mu
=\sum_{i=1}^NA^{-1}\gamma_i^*\mu_i.
\end{equation}

For the Stokes problem, the pressure is recovered, up to a constant, from $\nabla p=\Delta u$ on $\Omega_{\mathrm p}$ after computing $u$.

The following proposition establishes solvability of the dual problem and relates its energy to that of the primal problem.

\begin{proposition}
\label{Prop:dual}
The operator $S\colon W^*\to W$ is SPD. The dual problem~\eqref{dual_energy} has a unique solution $\lambda$ characterized by~\eqref{dual_equation}, and $u(\lambda)$ is the solution of~\eqref{primal_energy}. For any $\mu,r\in W^*$, the associated primal variables defined in~\eqref{linear_recovery} satisfy
\begin{equation*}
a(u(\mu),u(r))=\langle r,S\mu\rangle=\langle \mu,Sr\rangle.
\end{equation*}
\end{proposition}

%

\subsection{Reflections as dual block corrections}
We relate block iterative methods for the dual equation~\eqref{dual_equation} to the reflection algorithms. Using the block structure~\eqref{S_blocks} of $S$, we write
\begin{equation*}
S=D+L+L^*,
\end{equation*}
where $D=\operatorname{diag}(S_{11},\ldots,S_{NN})$ and $L\colon W^*\to W$ is block strictly lower triangular.
The block Jacobi and block Gauss--Seidel methods are given by (see, e.g.,~\cite{XZ:2017})
\begin{subequations}
\label{dual_block_iterations}
\begin{align}
\label{dual_jacobi}
\lambda^{(n+1)}&=\lambda^{(n)}+\omega D^{-1}(g-S\lambda^{(n)}),\\
\label{dual_gauss_seidel}
\lambda^{(n+1)}&=\lambda^{(n)}+(D+L)^{-1}(g-S\lambda^{(n)}),
\end{align}
\end{subequations}
where $\omega>0$ is the relaxation parameter for the block Jacobi method.

As shown in~\cite{LLS:2021}, a reflection step is the $a(\cdot,\cdot)$-orthogonal projection onto the affine constraint set $K_i=\{v\in V:\gamma_i v=g_i\}$:
\[
\operatorname{proj}_{K_i}^{a}v
=v+\mathcal H_i(g_i-\gamma_i v),
\quad \mathcal H_i=A^{-1}\gamma_i^*S_{ii}^{-1}.
\]
Thus alternating and parallel reflections are instances of the cyclic and relaxed parallel projection methods for~\eqref{primal_energy}, as considered in~\cite{JPX:2025,XZ:2002}. The dualization relations between these projection methods and the block relaxation methods for the dual problem are established in~\cite[Theorems~5.1 and~5.7]{JPX:2025}; see also~\cite{GM:1989,Han:1988,XZ:2002} for earlier projection--block-method correspondences. These results give the following theorem.

\begin{theorem}
\label{Thm:dualization}
The alternating and parallel methods of reflections for~\eqref{primal_energy}, given in \cref{Alg:alternating_reflections,Alg:parallel_reflections}, are dualizations of the block Gauss--Seidel method~\eqref{dual_gauss_seidel} and the block Jacobi method~\eqref{dual_jacobi} for~\eqref{dual_energy}, respectively, with the same relaxation parameter in the parallel methods. More precisely, if
\[
u^{(0)}=A^{-1}\gamma^*\lambda^{(0)},
\]
then we have
\[
u^{(n)}=A^{-1}\gamma^*\lambda^{(n)},
\quad n\ge1.
\]
\end{theorem}

By \cref{Prop:dual}, the corresponding primal and dual errors satisfy
\begin{equation}
\label{energy_error}
\|u^{(n)}-u\|_A^2
=\langle \lambda^{(n)}-\lambda,S(\lambda^{(n)}-\lambda)\rangle
=\|\lambda^{(n)}-\lambda\|_S^2.
\end{equation}
Consequently, every dual error estimate in the $S$-norm gives the same primal energy-error estimate.

The equivalence in \cref{Thm:dualization} also allows the parallel method of reflections to be used as a preconditioner. For a current primal iterate $v=u(\mu)$, let $r=g-\gamma v=g-S\mu\in W$ be its boundary discrepancy. The block Jacobi preconditioning step computes $D^{-1}r$, whose corresponding primal correction is given in \cref{Alg:parallel_correction}. The use of reflections as preconditioners for Krylov methods is also discussed in~\cite{CGHS:2019}.

\begin{algorithm}
\caption{Parallel reflection correction for boundary discrepancy $r\in W$}
\label{Alg:parallel_correction}
\begin{algorithmic}
\FOR{$i=1,2,\ldots,N$ \textbf{in parallel}}
  \STATE Compute the correction using $\mathcal H_i$ (see~\eqref{harmonic_extension}):
\[
d_i=\mathcal H_i r_i.
\]
\ENDFOR
\STATE Return the combined correction
\[
d=\sum_{i=1}^N d_i.
\]
\end{algorithmic}
\end{algorithm}

Indeed, the correction given in \cref{Alg:parallel_correction} satisfies
\[
d=\sum_{i=1}^N A^{-1}\gamma_i^*S_{ii}^{-1}r_i=u(D^{-1}r).
\]
Since $D^{-1}$ is SPD, this correction can be combined with the scalar updates of the conjugate gradient method. The primal and dual errors are related by~\eqref{energy_error}.

\begin{remark}
\label{Rem:boundary_formulation}
The boundary formulation introduced in~\cite{CGHS:2019} identifies alternating and parallel reflections with the block Gauss--Seidel and block Jacobi methods, respectively. It is related to the dual problem~\eqref{dual_energy} by the block diagonal change of variables $w=D\lambda$. Here $w_i=S_{ii}\lambda_i$ is the trace of $A^{-1}\gamma_i^*\lambda_i$ on $\Gamma_i$. The component traces satisfy
\begin{equation}
\label{boundary_system}
SD^{-1}w=g,
\quad (SD^{-1})_{ij}=\gamma_i\mathcal H_j,
\end{equation}
which, for the Laplace problem, is the boundary system in~\cite[Equations~(3.13)--(3.14)]{CGHS:2019}. The same change of variables carries the block iterations~\eqref{dual_block_iterations} to the corresponding iterations for~\eqref{boundary_system}.

\end{remark}

\section{Geometric assumptions and one-level methods}
\label{Sec:OneLevel}
In this section, we analyze how the convergence rates of the methods of reflections depend on the geometry of the perforated domain, in particular on the number, sizes, and separation of the sets $O_i$. By \cref{Thm:dualization}, these methods are equivalent to the corresponding block methods for the dual problem, and~\eqref{energy_error} identifies their error norms. It therefore suffices to analyze the block methods. The estimates below apply to both model problems. We first formulate the geometric assumptions required for the analysis and then derive convergence estimates under these assumptions.

\subsection{Geometric assumptions}
\label{Subsec:geometry}
To investigate the influence of the domain geometry on convergence, we first describe its relevant properties mathematically; see, e.g.,~\cite{RS:2021} for this approach in the analysis of Schwarz methods. We require each $O_i$ to lie in a ball that can be enlarged by a fixed factor while remaining inside $\Omega$ and disjoint from the other enlarged balls. This condition expresses separation relative to the sizes of the sets. It covers regular arrays of well-separated holes and their perturbations, with unequal sizes and nonspherical shapes permitted.

\begin{assumption}
\label{Ass:neutral}
There are centers $x_i$, radii $a_i>0$, and a parameter $0<\beta\le1$, such that
\begin{equation*}
\overline{O}_i\subseteq\overline{B(x_i,a_i)},
\quad U_i=B(x_i,(1+\beta)a_i)\subset\Omega,
\quad U_i\cap U_j=\emptyset\quad(i\ne j).
\end{equation*}
The restriction $\beta\le1$ entails no loss of generality, since larger neighborhoods may be reduced.
\end{assumption}

Unless otherwise stated, the constants hidden in $\lesssim$ and $\gtrsim$ are independent of $\beta$, $N$, $\operatorname{diam}\Omega$, and the radii. To write the estimates for both models, we define
\begin{equation}
\label{model_exponent}
\nu=\begin{cases}
1,&\text{for the Laplace problem},\\
3,&\text{for the Stokes problem}.
\end{cases}
\end{equation}

We first recall a local consequence of the Sobolev embedding.

\begin{lemma}
\label{Lem:local_L2}
Suppose that \cref{Ass:neutral} holds. For every $v\in H_0^1(\Omega)$ and $1\le i\le N$, we have
\[
\|v\|_{L^2(U_i)}\lesssim a_i|v|_{H^1(\Omega)}.
\]
\end{lemma}
\begin{proof}
Extend $v$ by zero outside $\Omega$. The H\"older inequality and the Sobolev embedding give
\[
\|v\|_{L^2(U_i)}
\le |U_i|^{\frac{1}{3}}\|v\|_{L^6(\mathbb R^3)}
\lesssim (1+\beta)a_i|v|_{H^1(\Omega)}
\le 2a_i|v|_{H^1(\Omega)}.
\]
This completes the proof.
\end{proof}

The following lemma estimates the $L^2$-norm in the thin region $U_i\setminus\overline{B(x_i,a_i)}$. The factor $\beta$ in this estimate controls the terms arising when derivatives fall on a cutoff that changes across this region. Related trace theorem-type arguments for $H^1$-functions appear in~\cite{DW:1994} and~\cite[Lemma~3.10]{TW:2005}, and an extension to $H^m$-functions is given in~\cite[Lemma~4.2]{Park:2024}.

\begin{lemma}
\label{Lem:boundary_layer}
Suppose that \cref{Ass:neutral} holds. For $m\in\{1,2\}$, every $v\in H^m(U_i)$, and $1\le i\le N$, we have
\begin{equation*}
\|v\|_{L^2(U_i\setminus\overline{B(x_i,a_i)})}^2
\lesssim\beta\bigl(\|v\|_{L^2(U_i)}^2+a_i^{2m}|v|_{H^m(U_i)}^2\bigr).
\end{equation*}
\end{lemma}
\begin{proof}
We first consider $m=1$. Partition $\partial U_i$ into shape-regular curved triangles of diameter comparable to $\beta a_i$. Their radial extensions to $\partial B(x_i,a_i)$ form uniformly shape-regular patches $P_1,\ldots,P_J$ of $U_i\setminus\overline{B(x_i,a_i)}$. The Friedrichs inequality with scaling~\cite[Lemma~A.17]{TW:2005} gives
\[
\|v\|_{L^2(P_j)}^2
\lesssim\beta^2a_i^2|v|_{H^1(P_j)}^2
+\beta a_i\|v\|_{L^2(\partial P_j\cap\partial U_i)}^2,
\quad 1\le j\le J.
\]
Summing over $j=1,\ldots,J$ gives
\begin{equation}
\label{layer_friedrichs}
\|v\|_{L^2(U_i\setminus\overline{B(x_i,a_i)})}^2
\lesssim\beta^2a_i^2|v|_{H^1(U_i\setminus\overline{B(x_i,a_i)})}^2
+\beta a_i\|v\|_{L^2(\partial U_i)}^2.
\end{equation}
Since the radius of $U_i$ lies between $a_i$ and $2a_i$, the trace inequality and scaling give
\[
\|v\|_{L^2(\partial U_i)}^2
\lesssim\frac{1}{a_i}\|v\|_{L^2(U_i)}^2+a_i|v|_{H^1(U_i)}^2.
\]
Substituting this bound into~\eqref{layer_friedrichs} and using $\beta^2\le\beta$ proves the case $m=1$. For $m=2$, the Sobolev interpolation inequality~\cite[Theorem~5.2]{AF:2003} on a reference ball and scaling give
\[
a_i^2|v|_{H^1(U_i)}^2
\lesssim\|v\|_{L^2(U_i)}^2+a_i^4|v|_{H^2(U_i)}^2.
\]
Combining this inequality with the case $m=1$ completes the proof.
\end{proof}

The next lemma constructs functions in $V$ that agree with given functions on $B(x_i,a_i)$ and vanish outside $U_i$.

\begin{lemma}
\label{Lem:localization}
Suppose that \cref{Ass:neutral} holds. For each $1\le i\le N$, let $v\in H^1(U_i)$ for the Laplace problem, or let $v\in H^1(U_i)^3$ with $\operatorname{div}v=0$ for the Stokes problem. There is a linear operator $\mathcal R_i$ from the corresponding space into $V$ such that, for every such $v$, $\mathcal R_i v$ vanishes outside $U_i$, equals $v$ on $B(x_i,a_i)$, and satisfies
\begin{equation*}
|\mathcal R_i v|_{H^1(U_i)}^2\lesssim\beta^{-\nu}\left(|v|_{H^1(U_i)}^2+\frac{\|v\|_{L^2(U_i)}^2}{a_i^2}\right),
\end{equation*}
where $\nu$ is defined in~\eqref{model_exponent}.
\end{lemma}
\begin{proof}
Choose a smooth cutoff~\cite[Appendix~C.5]{Evans:2010} $\chi_i$ equal to one on $B(x_i,a_i)$ and supported in $U_i$, with
\begin{equation}
\label{cutoff_bounds}
0\le\chi_i\le1,
\quad \|\nabla\chi_i\|_{L^\infty}\lesssim\frac1{\beta a_i},
\quad \|\nabla^2\chi_i\|_{L^\infty}\lesssim\frac1{\beta^2a_i^2}.
\end{equation}
For the Laplace problem, set $\mathcal R_i v=\chi_i v$. Since $\nabla\chi_i$ vanishes on $B(x_i,a_i)$, the product rule and \cref{Lem:boundary_layer} with $m=1$ give
\[
\begin{aligned}
|\chi_i v|_{H^1(U_i)}^2
&\le2\|\chi_i\nabla v\|_{L^2(U_i)}^2+2\|v\nabla\chi_i\|_{L^2(U_i)}^2\\
&\lesssim |v|_{H^1(U_i)}^2+\frac{\|v\|_{L^2(U_i\setminus\overline{B(x_i,a_i)})}^2}{\beta^2a_i^2}\\
&\lesssim\frac{1}{\beta}\left(|v|_{H^1(U_i)}^2+\frac{\|v\|_{L^2(U_i)}^2}{a_i^2}\right).
\end{aligned}
\]

For the Stokes problem, the vector potential estimate in~\cite[Proposition~4.1]{CM:2010}, with scaling~\cite[Section~3.4]{TW:2005}, gives a function $\psi_i\in H^2(U_i)^3$, depending linearly on $v$, such that $\operatorname{curl}\psi_i=v$ and
\begin{equation}
\label{stokes_vector_potential}
\frac{\|\psi_i\|_{L^2(U_i)}^2}{a_i^4}
+\frac{|\psi_i|_{H^1(U_i)}^2}{a_i^2}
+|\psi_i|_{H^2(U_i)}^2
\lesssim |v|_{H^1(U_i)}^2+\frac{\|v\|_{L^2(U_i)}^2}{a_i^2}.
\end{equation}
We define $\mathcal R_i v$ by
\begin{equation}
\label{stokes_extension}
\mathcal R_i v=\operatorname{curl}(\chi_i\psi_i)
=\chi_i v+\nabla\chi_i\times\psi_i.
\end{equation}
This function is divergence-free, vanishes outside $U_i$, and agrees with $v$ on $B(x_i,a_i)$. Applying \cref{Lem:boundary_layer} with $m=2$ componentwise to $\psi_i$, we obtain
\begin{equation}
\label{stokes_potential_layer}
\|\psi_i\|_{L^2(U_i\setminus\overline{B(x_i,a_i)})}^2
\lesssim\beta\bigl(\|\psi_i\|_{L^2(U_i)}^2+a_i^4|\psi_i|_{H^2(U_i)}^2\bigr).
\end{equation}
Since the derivatives of $\chi_i$ vanish on $B(x_i,a_i)$, differentiating~\eqref{stokes_extension} and using~\eqref{cutoff_bounds} and~\eqref{stokes_potential_layer} give
\[
\begin{aligned}
|\mathcal R_i v|_{H^1(U_i)}^2
&\lesssim |\chi_i\psi_i|_{H^2(U_i)}^2\\
&\lesssim |\psi_i|_{H^2(U_i)}^2
+\frac{|\psi_i|_{H^1(U_i)}^2}{\beta^2a_i^2}
+\frac{\|\psi_i\|_{L^2(U_i\setminus\overline{B(x_i,a_i)})}^2}{\beta^4a_i^4}\\
&\lesssim\beta^{-3}\left(|\psi_i|_{H^2(U_i)}^2
+\frac{|\psi_i|_{H^1(U_i)}^2}{a_i^2}
+\frac{\|\psi_i\|_{L^2(U_i)}^2}{a_i^4}\right).
\end{aligned}
\]
Applying~\eqref{stokes_vector_potential} proves the Stokes estimate and completes the proof.
\end{proof}

\subsection{One-level methods}
\label{Subsec:onelevel}
We now apply the geometric estimates to the block methods~\eqref{dual_block_iterations}. Let $I_i\colon W_i^*\to W^*$ be the canonical injection. Its dual adjoint $I_i^*\colon W\to W_i$ extracts the $i$th component. For the block diagonal operator $D$, we have
\[
\|\mu\|_D^2=\sum_{i=1}^N\|I_i\mu_i\|_S^2=\sum_{i=1}^N\|\mu_i\|_{S_{ii}}^2,
\quad \mu\in W^*.
\]
The following lemma compares the $S$- and $D$-norms. Similar arguments can be found in~\cite[Lemmas~4.7 and~4.8]{HS:2021}.

\begin{lemma}
\label{Lem:full_energy}
Suppose that \cref{Ass:neutral} holds. For every $\mu\in W^*$, we have
\begin{subequations}
\begin{align}
\|\mu\|_S^2&\lesssim \min\{N,\beta^{-\nu}N^{\frac{2}{3}}\}\|\mu\|_D^2,
\label{full_energy_bounds}\\
\|\mu\|_D^2&\lesssim\beta^{-\nu}\|\mu\|_S^2,
\label{full_energy_lower}
\end{align}
\end{subequations}
where $\nu$ is defined in~\eqref{model_exponent}.
\end{lemma}
\begin{proof}
Set $u_i=A^{-1}\gamma_i^*\mu_i$. The Cauchy--Schwarz inequality gives
\[
\|\mu\|_S^2=\left\|\sum_{i=1}^N u_i\right\|_A^2
\le N\sum_{i=1}^N\|u_i\|_A^2=N\|\mu\|_D^2.
\]
For the other upper bound, the disjointness of the neighborhoods, the H\"older inequality, and the Sobolev embedding give
\[
\sum_{i=1}^N a_i^{-2}\|v\|_{L^2(U_i)}^2
\lesssim\sum_{i=1}^N\|v\|_{L^6(U_i)}^2
\le N^{\frac{2}{3}}\|v\|_{L^6(\Omega)}^2
\lesssim N^{\frac{2}{3}}|v|_{H^1(\Omega)}^2.
\]
The operators $\mathcal R_i$ from \cref{Lem:localization} therefore satisfy
\[
\sum_{i=1}^N|\mathcal R_i v|_{H^1(U_i)}^2\lesssim \beta^{-\nu}N^{\frac{2}{3}}|v|_{H^1(\Omega)}^2.
\]
Since $\gamma_i(\mathcal R_i v)=\gamma_i v$, we have
\begin{equation}
\label{localized_duality}
|\langle \mu,\gamma v\rangle|
=\left|\sum_{i=1}^N a(u_i,\mathcal R_i v)\right|
\le\|\mu\|_D\left(\sum_{i=1}^N|\mathcal R_i v|_{H^1(U_i)}^2\right)^{\frac{1}{2}}.
\end{equation}
Taking the supremum over $\|v\|_A=1$ proves~\eqref{full_energy_bounds}.

To prove~\eqref{full_energy_lower}, apply \cref{Lem:local_L2,Lem:localization} to obtain
\[
|\mathcal R_i u_i|_{H^1(U_i)}^2\lesssim\beta^{-\nu}\|\mu_i\|_{S_{ii}}^2.
\]
Since the supports of the $\mathcal R_i u_i$ are disjoint and $\gamma_i(\mathcal R_i u_i)=S_{ii}\mu_i$, we have
\[
\|\mu\|_D^2
=\sum_{i=1}^N a(u(\mu),\mathcal R_i u_i)
\le\|\mu\|_S\left(\sum_{i=1}^N|\mathcal R_i u_i|_{H^1(U_i)}^2\right)^{\frac{1}{2}}
\lesssim\beta^{-\frac{\nu}{2}}\|\mu\|_S\|\mu\|_D.
\]
This proves~\eqref{full_energy_lower} and completes the proof.
\end{proof}

Since $D^{-1}$ is SPD, it can be used as a preconditioner for the conjugate gradient method; see~\cite{GHS:2014} for discussions of the conjugate gradient method in Hilbert spaces. We therefore estimate the extremal spectral values of $D^{-1}S$ and its condition number. The operator $D^{-1}S$ is self-adjoint and coercive in both the $D$- and $S$-inner products.

From \cref{Lem:full_energy}, we have the following theorem.

\begin{theorem}
\label{Thm:one_jacobi}
Suppose that \cref{Ass:neutral} holds. Then we have
\[
\inf\sigma(D^{-1}S)\gtrsim\beta^{\nu},
\quad \sup\sigma(D^{-1}S)\lesssim\min\{N,\beta^{-\nu}N^{\frac{2}{3}}\},
\]
where $\nu$ is defined in~\eqref{model_exponent}. Consequently, we have
\[
\kappa(D^{-1}S)\lesssim
\min\{\beta^{-\nu}N,\beta^{-2\nu}N^{\frac{2}{3}}\}.
\]
\end{theorem}

\begin{remark}
\label{Rem:one_jacobi_contraction}
The spectral theorem gives
\[
\|I-\omega D^{-1}S\|_S=\max\bigl\{|1-\omega\inf\sigma(D^{-1}S)|,\,|1-\omega\sup\sigma(D^{-1}S)|\bigr\}<1
\]
for $0<\omega<\frac{2}{\sup\sigma(D^{-1}S)}$. Hence the spectral estimates in \cref{Thm:one_jacobi} also give contraction estimates for \cref{Alg:parallel_reflections} through~\eqref{energy_error}.
\end{remark}

The preconditioner for the block Gauss--Seidel method is defined by
\begin{equation*}
B_{\mathrm{GS}}=(D+L)^{-1}.
\end{equation*}
The error propagation operator is $I-B_{\mathrm{GS}}S$~\cite{XZ:2017}. We estimate its energy norm.

\begin{theorem}
\label{Thm:one_gauss_seidel}
Suppose that \cref{Ass:neutral} holds. For $B_{\mathrm{GS}}=(D+L)^{-1}$, we have
\begin{equation*}
\|I-B_{\mathrm{GS}}S\|_S^2\le
1-\frac1{1+C\beta^{-\nu}N^2},
\end{equation*}
where $\nu$ is defined in~\eqref{model_exponent} and $C>0$ is independent of $\beta$, $N$, $\operatorname{diam}\Omega$, and the radii.
\end{theorem}
\begin{proof}
The Xu--Zikatanov identity~\cite{XZ:2002} (see also~\cite[Theorem~4.1]{XZ:2017}) gives
\begin{equation}
\label{one_gs}
\|I-B_{\mathrm{GS}}S\|_S^2
=1-\left(1+\sup_{v\ne0}\frac{\|L^*v\|_{D^{-1}}^2}{\|v\|_S^2}\right)^{-1}.
\end{equation}
For the upper bound, \cref{Prop:dual} gives $\|S_{ij}v_j\|_{S_{ii}^{-1}}\le\|v_j\|_{S_{jj}}$. Therefore
\[
\|L^*v\|_{D^{-1}}^2
\le\sum_{i=1}^N\left(\sum_{j=i+1}^N\|v_j\|_{S_{jj}}\right)^2
\le\frac{N(N-1)}2\|v\|_D^2.
\]
By \cref{Lem:full_energy}, this gives
\[
\frac{\|L^*v\|_{D^{-1}}^2}{\|v\|_S^2}
\lesssim\beta^{-\nu}N^2.
\]
Substitution into~\eqref{one_gs} proves the stated estimates and completes the proof.
\end{proof}


\subsection{Block ordering and comparison of one-level methods}
\label{Subsec:ordering}
To examine how block ordering affects the convergence rate of \cref{Alg:alternating_reflections}, we use the symmetrization of the corresponding block Gauss--Seidel method~\cite{XZ:2017}, defined by
\begin{equation*}
\bar B_{\mathrm{GS}}
=B_{\mathrm{GS}}+B_{\mathrm{GS}}^*
-B_{\mathrm{GS}}^*SB_{\mathrm{GS}}
=(D+L)^{-*}D(D+L)^{-1}.
\end{equation*}
Since $L^*$ has a nontrivial kernel, $\sup\sigma(\bar B_{\mathrm{GS}}S)=1$, and~\eqref{one_gs} gives
\begin{equation}
\label{sgs_comparison_identity}
\kappa(\bar B_{\mathrm{GS}}S)
=1+\sup_{v\ne0}\frac{\|L^*v\|_{D^{-1}}^2}{\|v\|_S^2}
=\bigl(1-\|I-B_{\mathrm{GS}}S\|_S^2\bigr)^{-1}.
\end{equation}
Thus the condition number of the symmetrized preconditioner determines the energy contraction factor of the alternating method of reflections.

The convergence rate of the Gauss--Seidel method can depend strongly on the ordering of the blocks. The Gauss--Seidel and symmetric Gauss--Seidel methods can even converge more slowly than the relaxed Jacobi method for certain SPD systems with equal diagonal entries~\cite{SY:2021,XXS:2021}.

We next examine block ordering for the Laplace problem on cubic arrays of $N=m^3$ spheres of radius $0.3$ and lattice spacing one, inside $\Omega=B(0,2m)$, with $m=2,3,4,5$. These configurations satisfy \cref{Ass:neutral} with $\beta=\frac12$. We denote the discrete SPD matrix by $S_h$; implementation details are given in \cref{Sec:Computational}. Its block diagonal part is $D_h$, and $\bar B_{\mathrm{GS},h}$ is the symmetrized block Gauss--Seidel preconditioner. The lexicographic order traverses the three grid indices with the third index varying fastest. The checkerboard order visits all centers with even sum of the grid indices before those with odd sum, preserving lexicographic order within each group.

\begin{table}
\centering
\begin{tabular}{r|rrr}
$N$ & Jacobi & SGS (lex.) & SGS (check.)\\ \hline
8 & 4.003 & 1.514 & 1.756 \\
27 & 8.512 & 2.596 & 4.309 \\
64 & 14.807 & 4.239 & 10.373 \\
125 & 22.905 & 6.412 & 22.429 \\
\end{tabular}
\caption{Spectral condition numbers of $D_h^{-1}S_h$ and $\bar B_{\mathrm{GS},h}S_h$ for the Laplace problem on cubic arrays of spheres. SGS denotes symmetric block Gauss--Seidel, with lexicographic or checkerboard ordering.}
\label{Table:ordering}
\end{table}

\Cref{Table:ordering} shows that the checkerboard order gives larger condition numbers than the lexicographic order, with an increasing ratio as $N$ grows. For $N=125$, the checkerboard condition number is approximately $3.5$ times the lexicographic one and is close to that of the block Jacobi method. By~\eqref{sgs_comparison_identity}, these differences also determine the energy contraction factors of the corresponding alternating reflections.

Strategies for choosing favorable orders have been studied; see, e.g.,~\cite{LW:2019,OZ:2017}. For the boundary blocks considered here, an ordering with a satisfactory convergence bound is not apparent. Since successive corrections are sensitive to ordering and a suitable choice is unclear, we use parallel local block corrections in the two-level constructions below.

\section{Two-level methods}
\label{Sec:TwoLevel}
As shown in \cref{Sec:OneLevel}, the convergence rates of \cref{Alg:alternating_reflections,Alg:parallel_reflections} can deteriorate as $N$ increases. In this section, we construct and analyze a two-level method of reflections to address this dependence (cf.~\cite{CGHS:2019}). We first introduce a coarse space for the dual equation~\eqref{dual_equation} and derive estimates for dual variables with zero boundary means under \cref{Ass:neutral}. We then combine parallel local corrections with a coarse correction before and after them. The construction applies to both models, with scalar constant traces for the Laplace problem and constant vector traces for the Stokes problem. We describe the corresponding corrections of the primal variable and prove that the resulting SPD preconditioner has a condition number bounded independently of $N$, $\operatorname{diam}\Omega$, and the radii.

\subsection{Coarse space}
\label{Subsec:coarse}
We construct a coarse space from constant traces on the individual boundary components. Let $W_0$ be the space of tuples $\xi=(\xi_1,\ldots,\xi_N)$, where $\xi_i$ is a scalar constant for the Laplace problem and a constant vector in $\mathbb R^3$ for the Stokes problem. We identify $W_0$ with $\mathbb R^N$ and $\mathbb R^{3N}$, respectively, and use the Euclidean inner product. Constant vectors belong to the Stokes trace spaces in~\eqref{stokes_spaces}.

Define the coarse injection $\Pi_0\colon W_0\to W^*$ by
\begin{equation*}
(\Pi_0\xi)_i=S_{ii}^{-1}(\xi_i|_{\Gamma_i}),
\quad 1\le i\le N.
\end{equation*}
By the symmetry of $S_{ii}$, its dual adjoint satisfies
\[
(\Pi_0^*D\mu,\xi)_{W_0}
=\sum_{i=1}^N\langle \mu_i,\xi_i|_{\Gamma_i}\rangle,
\quad \mu\in W^*,\quad \xi\in W_0.
\]
Thus $\Pi_0^*D\mu=0$ means that $\mu_i$ annihilates every constant trace, componentwise for Stokes problems. The associated primal coarse functions are $\mathcal H_i1_{\Gamma_i}$ for the Laplace problem and $\mathcal H_i(e_j|_{\Gamma_i})$, $1\le j\le3$, for the Stokes problem, where $e_j$ are the coordinate vectors of $\mathbb R^3$.

The coarse correction $B_0\colon W\to W^*$ is defined by
\begin{equation}
\label{coarse}
B_0=\Pi_0(\Pi_0^*S\Pi_0)^{-1}\Pi_0^*.
\end{equation}
Then $B_0S$ is the $S$-orthogonal projection onto $\Pi_0W_0$. The coarse matrix has dimension $N$ for the Laplace problem and $3N$ for the Stokes problem.

\begin{remark}
\label{Rem:coarse_comparison}
For the Laplace problem, we compare the proposed coarse correction~\eqref{coarse} with the constant-mode boundary correction in~\cite{CGHS:2019}. The latter uses the $L^2$ pairing on the boundaries $\Gamma_i$. Its restriction is defined by
\[
R\colon W\to\mathbb R^N,
\quad (Rw)_i=\int_{\Gamma_i}w_i\,ds.
\]
Transforming the boundary correction by $w=D\lambda$ gives
\begin{equation*}
B_{0,\mathrm{tr}}=\Pi_0(RS\Pi_0)^{-1}R,
\end{equation*}
whenever the coarse matrix is nonsingular. Our coarse correction~\eqref{coarse} preserves the symmetry of the dual formulation and gives an $S$-orthogonal projection, which facilitates the convergence analysis.
\end{remark}

The following estimates control the components with zero boundary means and remain valid after minimization over the coarse component. 
Related decompositions separating constant and zero-mean boundary components are used for dielectric problems in~\cite{HS:2021} and for Stokes boundary operators in~\cite{SS:2014}.

\begin{lemma}
\label{Lem:neutral_geometry}
Suppose that \cref{Ass:neutral} holds. Let $z\in W^*$ have zero mean on each boundary component, that is,
\begin{equation}
\label{neutral_condition}
\Pi_0^*Dz=0
\quad\Longleftrightarrow\quad
\langle z_i,c_i\rangle=0\text{ for every constant }c_i\in W_i,\quad 1\le i\le N.
\end{equation}
Then we have
\begin{subequations}

\begin{align}
\|z\|_S^2&\lesssim\beta^{-\nu}\|z\|_D^2,
\label{neutral_upper}\\
\|z\|_D^2&\lesssim\beta^{-\nu}\min_{\xi\in W_0}\|z+\Pi_0\xi\|_S^2,
\label{neutral_lower}
\end{align}
\end{subequations}
where $\nu$ is defined in~\eqref{model_exponent}.
\end{lemma}
\begin{proof}
Let $\mathcal R_i$ be the operators defined in \cref{Lem:localization}. For $v\in V$, we denote its average over $U_i$ by
\[
\bar v^{\,U_i}=\frac{1}{|U_i|}\int_{U_i}v\,dx.
\]
By \cref{Lem:localization}, we have
\begin{equation}
\label{mean_localization}
|\mathcal R_i(v-\bar v^{\,U_i})|_{H^1(U_i)}^2\lesssim\beta^{-\nu}\left(|v|_{H^1(U_i)}^2+\frac{\|v-\bar v^{\,U_i}\|_{L^2(U_i)}^2}{a_i^2}\right)\lesssim\beta^{-\nu}|v|_{H^1(U_i)}^2,
\end{equation}
where the last inequality follows from the Poincar\'e inequality with a scaling argument.

To prove~\eqref{neutral_upper}, we use~\eqref{neutral_condition} to obtain
\[
\langle z_i,\gamma_i v\rangle
=\langle z_i,\gamma_i\bigl(\mathcal R_i(v-\bar v^{\,U_i})\bigr)\rangle.
\]
Proceeding as in~\eqref{localized_duality} and using~\eqref{mean_localization} together with the disjointness of the $U_i$, we obtain
\[
|\langle z,\gamma v\rangle|\le\|z\|_D\left(\sum_{i=1}^N|\mathcal R_i(v-\bar v^{\,U_i})|_{H^1(U_i)}^2\right)^{\frac{1}{2}}.
\]
Taking the supremum over $\|v\|_A=1$ and applying~\eqref{mean_localization} proves~\eqref{neutral_upper}.

To prove~\eqref{neutral_lower}, we use the $D$-orthogonality in~\eqref{neutral_condition} and \cref{Lem:full_energy} to obtain
\[
\|z\|_D^2\le\|z+\Pi_0\xi\|_D^2
\lesssim\beta^{-\nu}\|z+\Pi_0\xi\|_S^2,
\quad \xi\in W_0.
\]
Taking the minimum over $\xi$ completes the proof.
\end{proof}

The following stable decomposition is a direct consequence of \cref{Lem:neutral_geometry}.

\begin{lemma}
\label{Lem:neutral_stable}
Suppose that \cref{Ass:neutral} holds. Every $v\in W^*$ has the decomposition
\begin{equation*}
v=\Pi_0\xi_{D}+\sum_{i=1}^NI_iz_i,
\end{equation*}
where $\xi_D=(\Pi_0^*D\Pi_0)^{-1}\Pi_0^*Dv$ and $z=v-\Pi_0\xi_D$. Then $z$ satisfies~\eqref{neutral_condition}, and we have
\begin{equation*}
\|\Pi_0\xi_{D}\|_S^2+\sum_{i=1}^N\|I_iz_i\|_S^2\lesssim\beta^{-2\nu}\|v\|_S^2,
\end{equation*}
where $\nu$ is defined in~\eqref{model_exponent}.
\end{lemma}

\subsection{Hybrid two-level method}
\label{Subsec:hybrid_twolevel}
We apply the coarse correction before and after the parallel local corrections, giving the hybrid preconditioner
\begin{equation}
\label{hybrid}
B_{\mathrm{hyb}}=B_0+(I-B_0S)D^{-1}(I-SB_0).
\end{equation}
The corresponding preconditioned operator has the hybrid Schwarz form in~\cite[equation~(2.11)]{TW:2005}. We first describe its implementation in the primal variable, using the recovery map~\eqref{linear_recovery}.

The primal coarse space consists of extensions of constant boundary traces, since
\[
u(\Pi_0\xi)=\sum_{i=1}^N\mathcal H_i(\xi_i|_{\Gamma_i}).
\]
For boundary data $r\in W$, define the coarse extension $\mathcal H_0\colon W\to H$ by
\begin{equation}
\label{primal_coarse}
\mathcal H_0r=\sum_{i=1}^N\mathcal H_i(\xi_i|_{\Gamma_i}),
\quad\text{where}\quad (\Pi_0^*S\Pi_0)\xi=\Pi_0^*r.
\end{equation}
Thus $\mathcal H_0r=u(B_0r)$, and $\mathcal H_0\gamma$ is the $a(\cdot,\cdot)$-orthogonal projection onto the primal coarse space. The extensions of the constant scalar or vector basis traces and the coarse matrix are computed once.

For a current primal iterate $v=u(\mu)$, its boundary discrepancy is $r=g-\gamma v=g-S\mu$. The two-level reflection correction is given in \cref{Alg:hybrid_reflections}. The first coarse correction is followed by parallel local extensions, and the final coarse correction removes their coarse component.

\begin{algorithm}
\caption{Two-level reflection correction for boundary discrepancy $r\in W$}
\label{Alg:hybrid_reflections}
\begin{algorithmic}
\STATE Compute the coarse correction $c_0=\mathcal H_0r$ (see~\eqref{primal_coarse}).
\STATE Compute the local corrections in parallel using $\mathcal H_i$ (see~\eqref{harmonic_extension}):
\[
d_i=\mathcal H_i(r_i-\gamma_i c_0),\quad 1\le i\le N.
\]
\STATE Apply the final coarse correction and return
\[
d=c_0+\sum_{i=1}^N d_i-\mathcal H_0\gamma\left(\sum_{i=1}^N d_i\right).
\]
\end{algorithmic}
\end{algorithm}

Indeed, $\mathcal H_i=A^{-1}\gamma_i^*S_{ii}^{-1}$ gives
\[
\sum_{i=1}^N d_i=u\bigl(D^{-1}(I-SB_0)r\bigr),
\quad d=u\bigl(B_0r+(I-B_0S)D^{-1}(I-SB_0)r\bigr)=u(B_{\mathrm{hyb}}r).
\]
Hence \cref{Alg:hybrid_reflections} is exactly the primal realization of the dual preconditioning step. 

In \cref{Thm:hybrid}, we show that the estimates in \cref{Lem:neutral_geometry} give a condition number bound independent of $N$.

\begin{theorem}
\label{Thm:hybrid}
Suppose that \cref{Ass:neutral} holds. The preconditioner $B_{\mathrm{hyb}}$ defined in~\eqref{hybrid} is SPD and satisfies
\[
\inf\sigma(B_{\mathrm{hyb}}S)\gtrsim\beta^{\nu},
\quad\sup\sigma(B_{\mathrm{hyb}}S)\lesssim\beta^{-\nu},
\]
where $\nu$ is defined in~\eqref{model_exponent}. Consequently, we have
\begin{equation*}
\kappa(B_{\mathrm{hyb}}S)\lesssim\beta^{-2\nu}.
\end{equation*}
\end{theorem}
\begin{proof}
Let $0\ne v\in\ker(\Pi_0^*S)$. The $D$-orthogonal decomposition in \cref{Lem:neutral_stable} gives
\[
\|Sv\|_{D^{-1}}
=\sup_{0\ne z\in W^*,\,\Pi_0^*Dz=0}\frac{|\langle z,Sv\rangle|}{\|z\|_D}
\le\sup_{0\ne z\in W^*,\,\Pi_0^*Dz=0}\frac{\|z\|_S}{\|z\|_D}\|v\|_S.
\]
For the reverse bound, choose $z=v-\Pi_0\xi_D$ as in \cref{Lem:neutral_stable}. Orthogonality gives
\[
\|v\|_S^2=\langle z,Sv\rangle\le\|z\|_D\|Sv\|_{D^{-1}}.
\]
The upper and lower estimates in \cref{Lem:neutral_geometry}, respectively, applied to these two inequalities yield
\[
\beta^{\nu}\|v\|_S^2\lesssim\|Sv\|_{D^{-1}}^2\lesssim\beta^{-\nu}\|v\|_S^2.
\]
For arbitrary $v\in W^*$, consider the $S$-orthogonal decomposition
\[
v=v_0+v_\perp,
\quad\text{where}\quad v_0=B_0Sv\in\Pi_0W_0,
\quad v_\perp\in\ker(\Pi_0^*S).
\]
Using~\eqref{hybrid}, we obtain
\[
\langle B_{\mathrm{hyb}}Sv,Sv\rangle
=\|v_0\|_S^2+\|Sv_\perp\|_{D^{-1}}^2,
\quad \|v\|_S^2=\|v_0\|_S^2+\|v_\perp\|_S^2.
\]
The preceding bounds and symmetry prove coercivity and the stated spectral estimates. This completes the proof.
\end{proof}

\begin{remark}
\label{Rem:separation}
The bounds in \cref{Thm:one_jacobi,Thm:one_gauss_seidel,Thm:hybrid} depend explicitly on the separation parameter $\beta$. In \cref{Sec:Computational}, we will present numerical results showing that the condition numbers increase as the boundary gap decreases. Nearly touching boundaries are important in the simulation of dense particle suspensions, where narrow gaps produce strong interactions and require accurate local resolution~\cite{BTB:2026,Hofer:2021}. Related effects on the methods of reflections are discussed in~\cite{CGHS:2019,Hofer:2021}. Developing preconditioners that remain robust in this regime is left for future work; joint corrections for closely spaced boundary components provide one possible direction~\cite{BTB:2026}.
\end{remark}

\begin{remark}
\label{Rem:parallel_twolevel}
One may instead apply the local and coarse corrections in parallel, giving
\begin{equation}
\label{additive}
B_{\mathrm{par}}=D^{-1}+B_0.
\end{equation}
Under \cref{Ass:neutral}, the positivity of $B_0$ and the $S$-orthogonality of $B_0S$ give
\[
\inf\sigma(B_{\mathrm{par}}S)\ge\inf\sigma(D^{-1}S),
\quad \sup\sigma(B_{\mathrm{par}}S)\le1+\sup\sigma(D^{-1}S).
\]
Together with \cref{Thm:one_jacobi}, these estimates give
\[
\kappa(B_{\mathrm{par}}S)
\lesssim\min\{\beta^{-\nu}N,\beta^{-2\nu}N^{\frac{2}{3}}\},
\]
where $\nu$ is defined in~\eqref{model_exponent}. The resulting upper bound still depends on $N$.
\end{remark}

\section{Computational aspects}
\label{Sec:Computational}
In this section, we describe conforming Galerkin discretization for both model problems and show that it preserves the hybrid estimates. We then compare the one-level and two-level preconditioners numerically for both models and investigate the dependence on the number and separation of the holes.

\subsection{Galerkin discretization}
\label{Subsec:discretization}
For each $i$, let $W_{i,h}^*$ be a finite-dimensional subspace of $W_i^*$, and set $W_h^*=\prod_{i=1}^NW_{i,h}^*\subset W^*$. To retain all constant boundary modes, we assume that
\begin{equation}
\label{discrete_coarse_condition}
\Pi_0^*D(W_h^*)=W_0.
\end{equation}
Since $\Pi_0^*D$ records the boundary integrals, \eqref{discrete_coarse_condition} holds if each $W_{i,h}^*$ contains the constant functions for the Laplace problem or the classes of constant vector fields for the Stokes problem. Identifying elements of $W_h^*$ with their coefficient vectors, we define the Galerkin matrix and right-hand side by
\[
w_h^tS_hv_h=\langle w_h,Sv_h\rangle,
\quad w_h^tg_h=\langle w_h,g\rangle,
\quad v_h,w_h\in W_h^*.
\]
By replacing $W^*$ in the dual minimization problem~\eqref{dual_energy} with $W_h^*$, we obtain a discrete minimization problem whose optimality condition is the following SPD linear system:
\begin{equation*}
S_h\lambda_h=g_h.
\end{equation*}
Let $D_h=\operatorname{diag}(S_{11,h},\ldots,S_{NN,h})$. We define the coarse injection $\Pi_{0,h}$ and correction $B_{0,h}$ by
\begin{equation*}
\begin{gathered}
(\Pi_{0,h}^tD_hv_h,\xi)_{W_0}=\sum_{i=1}^N\langle v_{i,h},\xi_i|_{\Gamma_i}\rangle,
\quad \xi\in W_0,\\
B_{0,h}=\Pi_{0,h}(\Pi_{0,h}^tS_h\Pi_{0,h})^{-1}\Pi_{0,h}^t.
\end{gathered}
\end{equation*}
The condition~\eqref{discrete_coarse_condition} makes $\Pi_{0,h}$ injective and the coarse matrix SPD. The columns of $\Pi_{0,h}$ are obtained by independent local solves.

We next verify that the estimates used to analyze the hybrid preconditioner are preserved by this discretization. Suppose that \cref{Ass:neutral} holds. The estimates in \cref{Lem:neutral_geometry} pass to these matrices: for every $z_h\in W_h^*$ with $\Pi_{0,h}^tD_hz_h=0$, we have
\begin{subequations}

\begin{align}
\|z_h\|_{S_h}^2&\lesssim\beta^{-\nu}\|z_h\|_{D_h}^2,
\label{discrete_neutral_upper}\\
\|z_h\|_{D_h}^2&\lesssim\beta^{-\nu}\min_{\xi\in W_0}\|z_h+\Pi_{0,h}\xi\|_{S_h}^2,
\label{discrete_neutral_lower}
\end{align}
\end{subequations}
where $\nu$ is defined in~\eqref{model_exponent}.
The Galerkin norms agree with the continuous norms on $W_h^*$, and the discrete zero-mean condition is the restriction of~\eqref{neutral_condition}. Thus \cref{Lem:neutral_geometry} gives~\eqref{discrete_neutral_upper}. For~\eqref{discrete_neutral_lower}, the $D_h$-orthogonality and \cref{Lem:full_energy} give directly
\[
\|z_h\|_{D_h}^2\le\|z_h+\Pi_{0,h}\xi\|_{D_h}^2
\lesssim\beta^{-\nu}\|z_h+\Pi_{0,h}\xi\|_{S_h}^2,
\quad \xi\in W_0.
\]

Consequently, the argument in \cref{Thm:hybrid} applies to the matrix preconditioner
\begin{equation*}
B_{\mathrm{hyb},h}=B_{0,h}+(I-B_{0,h}S_h)D_h^{-1}(I-S_hB_{0,h})
\end{equation*}
and gives the condition number bound
\begin{equation*}
\kappa(B_{\mathrm{hyb},h}S_h)\lesssim\beta^{-2\nu},
\end{equation*}
where $\nu$ is defined in~\eqref{model_exponent}.

\subsection{Spherical harmonic discretization}
\label{Subsec:spherical_harmonics}
We now construct spaces satisfying~\eqref{discrete_coarse_condition} on spherical boundaries $\Gamma_i=\partial B(x_i,a_i)$. Let $Y_\ell^k$ be the real spherical harmonics of degree $\ell$ and order $k$ on $\mathbb S^2=\partial B(0,1)$, obtained by taking real and imaginary parts of the complex harmonics in~\cite[Definition~2.1]{CV:2018}. On $\Gamma_i$, we write $\theta=\frac{x-x_i}{a_i}\in\mathbb S^2$. For an integer $\ell_i\ge0$, we define the Laplace discrete space by
\begin{equation}
\label{scalar_harmonic_space}
W_{i,h}^*=\operatorname{span}\{Y_\ell^k(\theta):0\le\ell\le\ell_i,\ -\ell\le k\le\ell\},
\end{equation}
which has dimension $(\ell_i+1)^2$.

For the Stokes problem, we use the vector spherical harmonics~\cite[Definition~2.2]{CV:2018} defined by
\[
\begin{aligned}
V_\ell^k(\theta)&=\nabla_{\mathbb S^2}Y_\ell^k(\theta)-(\ell+1)Y_\ell^k(\theta)\theta,\\
W_\ell^k(\theta)&=\nabla_{\mathbb S^2}Y_\ell^k(\theta)+\ell Y_\ell^k(\theta)\theta,\\
X_\ell^k(\theta)&=\theta\times\nabla_{\mathbb S^2}Y_\ell^k(\theta),
\end{aligned}
\]
where $\nabla_{\mathbb S^2}$ denotes the surface gradient. For an integer $\ell_i\ge1$, the discrete space is defined by
\begin{equation}
\label{vector_harmonic_space}
W_{i,h}^*=\operatorname{span}\{V_\ell^k(\theta),W_\ell^k(\theta),X_\ell^k(\theta):1\le\ell\le\ell_i,\ -\ell\le k\le\ell\},
\end{equation}
viewed in the quotient~\eqref{stokes_quotient}, and has dimension $3\ell_i(\ell_i+2)$. The omitted mode $V_0^0=-Y_0^0\theta$ is proportional to $n_i$. In both cases, the basis functions are normalized in $L^2(\Gamma_i)$.

The scalar mode $Y_0^0$ is constant, and the three vector modes $W_1^k$, $-1\le k\le1$, span the constant vectors because
\[
\nabla_{\mathbb S^2}(c\cdot\theta)+(c\cdot\theta)\theta=c,
\quad c\in\mathbb R^3.
\]
Thus both discrete spaces contain all constant modes and satisfy~\eqref{discrete_coarse_condition}.

\subsection{Dependence on \texorpdfstring{$N$}{N}}
\label{Subsec:numerics}
We present numerical results for both model problems to examine the dependence on $N$. We compare the one-level and two-level preconditioners through their condition numbers and their effect on the convergence of the conjugate gradient method.

We use the same cubic arrays as in the ordering experiment in \cref{Sec:OneLevel}: $N=m^3$ spheres of radius $a=0.3$ and lattice spacing one, inside $\Omega=B(0,2m)$, with $m=2,3,4,5$. These configurations satisfy \cref{Ass:neutral} with $\beta=\frac12$.

We use the spaces in~\eqref{scalar_harmonic_space} and~\eqref{vector_harmonic_space} with $\ell_i=4$ for every $i$, giving 25 and 72 unknowns per sphere for the Laplace and Stokes problems, respectively.

\Cref{Table:conditions} compares the one-level block Jacobi preconditioner $D_h^{-1}$ and the fully parallel two-level preconditioner $B_{\mathrm{par}}$ in~\eqref{additive} and the hybrid two-level preconditioner $B_{\mathrm{hyb}}$ in~\eqref{hybrid}. The parallel correction reduces the condition numbers but retains growth with $N$. The hybrid condition numbers remain below $1.42$ for Laplace and $2.03$ for Stokes over the tested range, consistently with \cref{Thm:hybrid}.

\begin{table}
\centering
\begin{tabular}{r|rrr|rrr}
 & \multicolumn{3}{c|}{Laplace} & \multicolumn{3}{c}{Stokes}\\
$N$ & \makecell{One-level\\Jacobi} & \makecell{Two-level\\parallel} & \makecell{Two-level\\hybrid} & \makecell{One-level\\Jacobi} & \makecell{Two-level\\parallel} & \makecell{Two-level\\hybrid}\\ \hline
8 & 4.003 & 3.768 & 1.282 & 5.155 & 4.109 & 1.642 \\
27 & 8.512 & 6.365 & 1.350 & 10.814 & 6.510 & 1.814 \\
64 & 14.807 & 9.959 & 1.398 & 19.146 & 9.762 & 1.944 \\
125 & 22.905 & 14.552 & 1.419 & 30.234 & 13.938 & 2.021 \\
\end{tabular}
\caption{Spectral condition numbers of the one-level and two-level preconditioned systems for cubic arrays of spheres.}
\label{Table:conditions}
\end{table}

For the conjugate gradient comparison, we take $N=125$ and prescribe the coefficient vector of the discrete solution by
\begin{equation}
\label{numerical_rhs_3d}
(\lambda_h)_j=\sin(\sqrt2\,j)+\frac14\cos(\sqrt3\,j),
\quad g_h=S_h\lambda_h,
\end{equation}
where the boundary components are indexed in lexicographic order. All methods start from zero. \Cref{Fig:sphere_cg_laplace,Fig:sphere_cg_stokes} show the relative energy errors $\frac{\|\lambda_h-\lambda_h^{(k)}\|_{S_h}}{\|\lambda_h\|_{S_h}}$. The hybrid preconditioner gives the fastest decrease for both models in these experiments.

\begin{figure}
\centering
\subfloat[Laplace condition numbers.\label{Fig:sphere_condition_laplace}]{%
\includegraphics[width=0.48\textwidth]{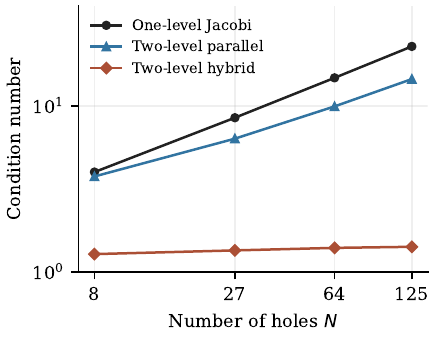}}
\quad
\subfloat[Stokes condition numbers.\label{Fig:sphere_condition_stokes}]{%
\includegraphics[width=0.48\textwidth]{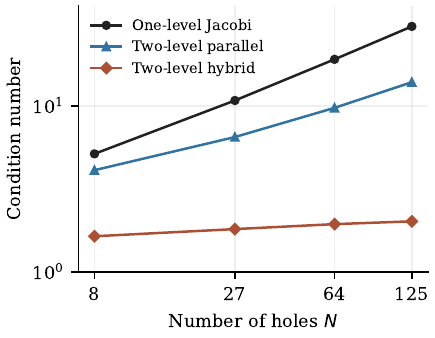}}\\
\subfloat[Laplace energy errors, $N=125$.\label{Fig:sphere_cg_laplace}]{%
\includegraphics[width=0.48\textwidth]{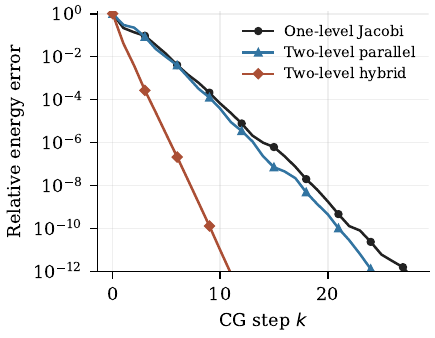}}
\quad
\subfloat[Stokes energy errors, $N=125$.\label{Fig:sphere_cg_stokes}]{%
\includegraphics[width=0.48\textwidth]{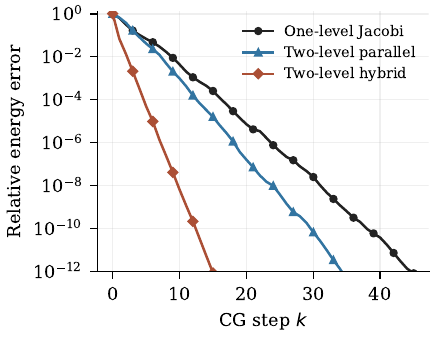}}
\caption{Comparison of the one-level Jacobi, fully parallel two-level, and hybrid two-level preconditioners. Panels (a) and (b) display the condition numbers in \cref{Table:conditions}; panels (c) and (d) show preconditioned conjugate gradient errors for~\eqref{numerical_rhs_3d}.}
\label{Fig:numerical_results}
\end{figure}

\subsection{Dependence on \texorpdfstring{$\beta$}{beta}}
We examine the separation dependence on two spheres centered at $(-\frac12,0,0)$ and $(\frac12,0,0)$ in $\Omega=B(0,4)$. For $\beta=2^{-j}$, $j=1,\ldots,6$, choose
\[
a=\frac1{2(1+2\beta)},
\quad \frac{\operatorname{dist}(\Gamma_1,\Gamma_2)}{a}=4\beta.
\]
The enlarged balls $U_i=B(x_i,(1+\beta)a)$, $i=1,2$, in \cref{Ass:neutral} remain disjoint.

For every value of $\beta$, we use the spaces in~\eqref{scalar_harmonic_space} and~\eqref{vector_harmonic_space} with the fixed degree $\ell_1=\ell_2=34$. The one-level and hybrid condition numbers are plotted in \cref{Fig:sphere_gap_laplace,Fig:sphere_gap_stokes}. Both increase as the gap decreases, with stronger deterioration for Stokes. These results illustrate the separation dependence discussed in \cref{Rem:separation}.

\begin{figure}
\centering
\subfloat[Laplace.\label{Fig:sphere_gap_laplace}]{%
\includegraphics[width=0.48\textwidth]{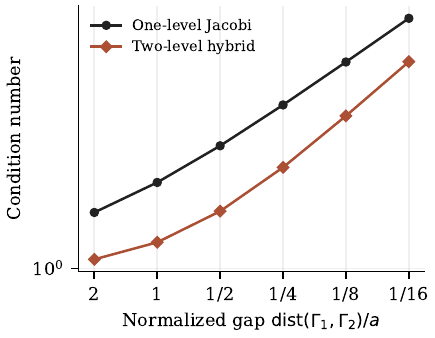}}
\quad
\subfloat[Stokes.\label{Fig:sphere_gap_stokes}]{%
\includegraphics[width=0.48\textwidth]{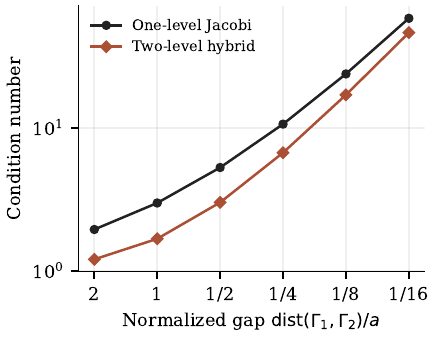}}
\caption{Condition numbers for two spheres with fixed harmonic degree 34 as the normalized boundary gap decreases.}
\label{Fig:sphere_gap_results}
\end{figure}

\section{Conclusion}
\label{Sec:Conclusion}
In this paper, we have developed and analyzed symmetric two-level preconditioners for the method of reflections for three-dimensional Laplace and Stokes problems. Their common dual formulation allows the same local and coarse correction algorithm to be used for both models. For fixed separation, the condition numbers are bounded independently of the number and sizes of the holes and the domain diameter. The analysis quantifies separation dependence through bounds of order $\beta^{-2\nu}$, where $\nu$ is defined in~\eqref{model_exponent}.

The results suggest several directions for future work. The coarse problems have $N$ and $3N$ unknowns for the Laplace and Stokes problems, respectively. Efficient approximate coarse solvers and multilevel extensions would therefore be useful for large configurations. Developing preconditioners that remain robust for nearly touching boundaries is another direction for future work.

\section*{Acknowledgments}
The author thanks Professor Martin J. Gander for introducing him to the method of reflections through his presentation at the 29th International Conference on Domain Decomposition Methods (DD29).

\bibliographystyle{siamplain}
\bibliography{refs_reflections}

\begin{thebibliography}{10}

\bibitem{AF:2003}
{\sc R.~A. Adams and J.~J.~F. Fournier}, {\em Sobolev Spaces}, Academic Press,
  Amsterdam, 2003.

\bibitem{BTB:2026}
{\sc A.~Broms, A.-K. Tornberg, and A.~H. Barnett}, {\em Preconditioning for
  near-contacts in large {2D} {Stokes} flows: a locally compressed method of
  fundamental solutions}, arXiv preprint arXiv:2606.11273,  (2026).

\bibitem{CDGH:2026a}
{\sc S.~Chaillat, M.~Darbas, M.~J. Gander, and L.~Halpern}, {\em Accelerating
  the method of reflections with domain decomposition techniques for boundary
  integral equations in multiple scattering}, BIT, 66 (2026).
\newblock Art.~57.

\bibitem{CDGH:2026b}
{\sc S.~Chaillat, M.~Darbas, M.~J. Gander, and L.~Halpern}, {\em Domain
  decomposition for integral equations: {Schwarz}, methods of reflections and
  elasticity}, in Domain Decomposition Methods in Science and Engineering XXIX,
  Springer, 2026, pp.~257--264.

\bibitem{CGHS:2019}
{\sc G.~Ciaramella, M.~J. Gander, L.~Halpern, and J.~Salomon}, {\em Methods of
  reflections: relations with {Schwarz} methods and classical stationary
  iterations, scalability and preconditioning}, SMAI J. Comput. Math., 5
  (2019), pp.~161--193.

\bibitem{CV:2018}
{\sc E.~Corona and S.~Veerapaneni}, {\em Boundary integral equation analysis
  for suspension of spheres in {Stokes} flow}, J. Comput. Phys., 362 (2018),
  pp.~327--345.

\bibitem{CM:2010}
{\sc M.~Costabel and A.~McIntosh}, {\em On {Bogovski\u{\i}} and regularized
  {Poincar\'e} integral operators for {de Rham} complexes on {Lipschitz}
  domains}, Math. Z., 265 (2010), pp.~297--320.

\bibitem{DW:1994}
{\sc M.~Dryja and O.~B. Widlund}, {\em Domain decomposition algorithms with
  small overlap}, SIAM J. Sci. Comput., 15 (1994), pp.~604--620.

\bibitem{ET:1999}
{\sc I.~Ekeland and R.~Temam}, {\em Convex Analysis and Variational Problems},
  SIAM, Philadelphia, 1999.

\bibitem{EG:2021}
{\sc A.~Ern and J.-L. Guermond}, {\em Finite Elements {I}: Approximation and
  Interpolation}, Springer, Cham, 2021.

\bibitem{Evans:2010}
{\sc L.~C. Evans}, {\em Partial Differential Equations}, American Mathematical
  Society, Providence, RI, 2010.

\bibitem{GM:1989}
{\sc N.~Gaffke and R.~Mathar}, {\em A cyclic projection algorithm via duality},
  Metrika, 36 (1989), pp.~29--54.

\bibitem{Galdi:2011}
{\sc G.~P. Galdi}, {\em An Introduction to the Mathematical Theory of the
  {Navier--Stokes} Equations: Steady-State Problems}, Springer, New York, 2011.

\bibitem{Golusin:1934}
{\sc G.~M. Golusin}, {\em {Aufl{\"o}sung einiger ebenen Grundaufgaben der
  mathematischen Physik im Fall der {Laplaceschen} Gleichung und
  mehrfachzusammenh{\"a}ngender Gebiete, die durch Kreise begrenzt sind}}, Mat.
  Sb., 41 (1934), pp.~246--276.

\bibitem{GHS:2014}
{\sc A.~G{\"u}nnel, R.~Herzog, and E.~Sachs}, {\em A note on preconditioners
  and scalar products in {Krylov} subspace methods for self-adjoint problems in
  {Hilbert} space}, Electron. Trans. Numer. Anal., 41 (2014), pp.~13--20.

\bibitem{Han:1988}
{\sc S.-P. Han}, {\em A successive projection method}, Math. Program., 40
  (1988), pp.~1--14.

\bibitem{HS:2021}
{\sc M.~Hassan and B.~Stamm}, {\em An integral equation formulation of the
  {$N$}-body dielectric spheres problem. {Part I}: numerical analysis}, ESAIM
  Math. Model. Numer. Anal., 55 (2021), pp.~S65--S102.

\bibitem{Hofer:2021}
{\sc R.~M. H{\"o}fer}, {\em Convergence of the method of reflections for
  particle suspensions in {Stokes} flows}, J. Differential Equations, 297
  (2021), pp.~81--109.

\bibitem{HV:2018}
{\sc R.~M. H{\"o}fer and J.~J.~L. Vel{\'a}zquez}, {\em The method of
  reflections, homogenization and screening for {Poisson} and {Stokes}
  equations in perforated domains}, Arch. Ration. Mech. Anal., 227 (2018),
  pp.~1165--1221.

\bibitem{JPX:2025}
{\sc B.~Jiang, J.~Park, and J.~Xu}, {\em Subspace correction as a general
  framework for convex optimization algorithms}, arXiv preprint
  arXiv:2505.09765,  (2025).

\bibitem{LLS:2021}
{\sc P.~Laurent, G.~Legendre, and J.~Salomon}, {\em On the method of
  reflections}, Numer. Math., 148 (2021), pp.~449--493.

\bibitem{LW:2019}
{\sc C.-P. Lee and S.~J. Wright}, {\em Random permutations fix a worst case for
  cyclic coordinate descent}, IMA J. Numer. Anal., 39 (2019), pp.~1246--1275.

\bibitem{McLean:2000}
{\sc W.~McLean}, {\em Strongly Elliptic Systems and Boundary Integral
  Equations}, Cambridge University Press, Cambridge, 2000.

\bibitem{OZ:2017}
{\sc P.~Oswald and W.~Zhou}, {\em Random reordering in {SOR}-type methods},
  Numer. Math., 135 (2017), pp.~1207--1220.

\bibitem{Park:2024}
{\sc J.~Park}, {\em Two-level overlapping {Schwarz} preconditioners with
  universal coarse spaces for {$2m$}th-order elliptic problems}, SIAM J. Sci.
  Comput., 46 (2024), pp.~A3681--A3702.

\bibitem{PX:2025}
{\sc J.~Park and J.~Xu}, {\em Auxiliary space theory for the analysis of
  iterative methods for semidefinite linear systems}, arXiv preprint
  arXiv:2509.07179,  (2025).

\bibitem{RS:2021}
{\sc A.~Reusken and B.~Stamm}, {\em Analysis of the {Schwarz} domain
  decomposition method for the conductor-like screening continuum model}, SIAM
  J. Numer. Anal., 59 (2021), pp.~769--796.

\bibitem{Rockafellar:1967}
{\sc R.~T. Rockafellar}, {\em Duality and stability in extremum problems
  involving convex functions}, Pacific J. Math., 21 (1967), pp.~167--187.

\bibitem{SS:2014}
{\sc F.-J. Sayas and V.~Selgas}, {\em Variational views of stokeslets and
  stresslets}, SeMA J., 63 (2014), pp.~65--90.

\bibitem{SY:2021}
{\sc R.~Sun and Y.~Ye}, {\em Worst-case complexity of cyclic coordinate
  descent: {$O(n^2)$} gap with randomized version}, Math. Program., 185 (2021),
  pp.~487--520.

\bibitem{TW:2005}
{\sc A.~Toselli and O.~B. Widlund}, {\em Domain Decomposition
  Methods---Algorithms and Theory}, Springer, Berlin, 2005.

\bibitem{XXS:2021}
{\sc P.~Xiao, Z.~Xiao, and R.~Sun}, {\em Two symmetrized coordinate descent
  methods can be {$O(n^2)$} times slower than the randomized version}, SIAM J.
  Optim., 31 (2021), pp.~2726--2752.

\bibitem{XZ:2002}
{\sc J.~Xu and L.~Zikatanov}, {\em The method of alternating projections and
  the method of subspace corrections in {Hilbert} space}, J. Amer. Math. Soc.,
  15 (2002), pp.~573--597.

\bibitem{XZ:2017}
{\sc J.~Xu and L.~Zikatanov}, {\em Algebraic multigrid methods}, Acta Numer.,
  26 (2017), pp.~591--721.

\end{thebibliography}
\end{document}